\documentclass[11pt,reqno]{amsart}

\usepackage[T1]{fontenc}
\usepackage{lmodern}
\usepackage{amsmath,amssymb}
\usepackage{microtype}
\usepackage[hidelinks]{hyperref}

\newtheorem{theorem}{Theorem}[section]
\newtheorem{lemma}[theorem]{Lemma}
\numberwithin{equation}{section}
\DeclareMathOperator{\Ran}{Ran}
\DeclareMathOperator{\re}{Re}
\DeclareMathOperator{\im}{Im}

\title[The Fong--Tsui conjecture]{A proof of the Fong--Tsui conjecture}

\author{Mohamed Amine Aouichaoui}
\address{Department of Mathematics, Faculty of Sciences of Monastir, University of Monastir, 5019 Monastir, Tunisia}
\email{amine.aouichaoui@fsm.rnu.tn}

\author{Fuad Kittaneh$^{*}$}
\address{Department of Mathematics, The University of Jordan, Amman, Jordan}
\address{Department of Mathematics, Korea University, 02841 Seoul, South Korea}
\email{fkitt@ju.edu.jo}
\thanks{$^{*}$ Corresponding author.}

\author{Yicen Ma}
\address{Zhejiang University}
\email{3250103237@zju.edu.cn}

\thanks{Additional Note: Mohamed Amine Aouichaoui and Yicen Ma are co-first authors.
The author list is ordered alphabetically.}


\begin{document}

\begin{abstract}
We prove that a bounded operator $T$ on a complex Hilbert space is
self-adjoint whenever $|T|\leq|\re T|$, as conjectured by Fong and Tsui \cite{FT}.
\end{abstract}

\maketitle

\section{Introduction}

Let $\mathcal H$ be a complex Hilbert space, and let $\mathcal B(\mathcal H)$
denote the algebra of bounded linear operators on $\mathcal H$. For
$T\in\mathcal B(\mathcal H)$, write
\[
 |T|=(T^*T)^{1/2},\qquad
 \re T=\frac{T+T^*}{2},\qquad
 \im T=\frac{T-T^*}{2i}.
\]
We use the usual order on self-adjoint operators and write $[S,R]=SR-RS$.

Fong and Tsui \cite[p.~74]{FT} conjectured that
\begin{equation}\label{eq:conjecture}
 |T|\leq|\re T|
\end{equation}
implies $T=T^*$. They proved this for matrices, compact operators,
and elements of finite von Neumann algebras
\cite[Theorem~5 and Remarks~2--3]{FT}. They also proved that
$|T|\leq\re T$ implies $T\geq0$.
Earlier, Fong and Istr\u{a}\c{t}escu \cite{FI} had shown that
$|T|^2\leq(\re T)^2$ implies $T=T^*$; see also Kittaneh \cite{K}.

Mortad \cite[Theorem~4]{M} proved the conjecture when $T^*T$ commutes
with $T+T^*$. Mbekhta and Suciu \cite{MS} proved it for partial isometries,
contractive $2$-quasi-isometries, and Brownian isometries of positive
covariance. Further sufficient conditions for binormal operators were
given by Stankovi\'c \cite{S}. To the best of our knowledge, the conjecture has remained open in its full generality. We prove the following theorem.

\begin{theorem}\label{thm:main}
Let $T\in\mathcal B(\mathcal H)$. If $|T|\leq|\re T|$, then $T=T^*$.
\end{theorem}

We shall give the proof of Theorem~\ref{thm:main} in the next section.

\section{An operator inequality and Proof of the main theorem}

We recall a standard fact about the equation $HZ+ZH=R$.
If $H=H^*\geq\delta I$ for some $\delta>0$, its unique bounded solution is
\begin{equation}\label{eq:lyapunov}
 Z=\int_0^\infty e^{-tH}R e^{-tH}\,dt;
\end{equation}
see \cite[Theorem~9.2]{BR}. The solution is self-adjoint when $R$ is self-adjoint and is positive
when $R$ is positive. In particular,
\begin{equation}\label{eq:order}
 Z=Z^*,\quad HZ+ZH\leq0\quad\Longrightarrow\quad Z\leq0.
\end{equation}

\begin{lemma}\label{lem:inequality}
Let $H,X,F\in\mathcal B(\mathcal H)$ satisfy
\[
 H=H^*\geq\delta I,\qquad X=X^*,\qquad F\geq0,
 \qquad \delta>0.
\]
If
\begin{equation}\label{eq:lemma-hypothesis}
 X^2+i[H,X]+HF+FH\leq0,
\end{equation}
then $X=F=0$.
\end{lemma}

\begin{proof}
Let $Y$ be the self-adjoint solution of $HY+YH=X$, and put
\[
 Z=2YHY+i[H,Y]+F.
\]
The identities $$(HY+YH)^2
   =H(2YHY)+(2YHY)H+[H,Y]^*[H,Y],$$ and $$H[H,Y]+[H,Y]H=[H,X];$$
give
\[
 HZ+ZH=X^2+i[H,X]+HF+FH-[H,Y]^*[H,Y]\leq0.
\]
Thus, $Z\leq0$ by \eqref{eq:order}.

\medskip

Now, put $V=I+2iY$. Direct multiplication gives
\[
 V^*HV=H+2i[H,Y]+4YHY=H+2Z-2F\leq H.
\]
By induction, $(V^*)^nHV^n\leq H$ for every integer $n\geq1$, and hence
\[
 \delta(V^*)^nV^n\leq\|H\|I.
\]
Since $V$ commutes with $V^*$ and $V^*V=I+4Y^2$, the binomial formula implies that
\[
 \delta(I+4nY^2)
 \leq\delta(I+4Y^2)^n
 =\delta(V^*)^nV^n
 \leq\|H\|I.
\]
Consequently,
\[
 4n\delta\|Yx\|^2\leq(\|H\|-\delta)\|x\|^2
 \qquad(x\in\mathcal H,\ n\geq1).
\]
It follows that $Y=0;$ and so $X=0$ and $Z=F\leq0$. Since $F\geq0$, we
also have $F=0$.
\end{proof}

\begin{theorem}\label{thm:quadratic}
Let $A,B\in\mathcal B(\mathcal H)$ be self-adjoint, and let
$K\in\mathcal B(\mathcal H)$ be positive. Suppose that, for some
$0<\alpha<2$,
\begin{equation}\label{eq:order-bound}
 K\leq\alpha|A|
\end{equation}
and
\begin{equation}\label{eq:quadratic}
 B^2+i[A,B]+|A|K+K|A|\leq K^2.
\end{equation}
Then, $B=K=0$.
\end{theorem}

\begin{proof}
Put $D=|A|$. Suppose that $k=\|K\|>0$, and choose $a$ such that
\begin{equation}\label{eq:threshold}
 \frac{k}{2}<a<\frac{k}{\alpha}.
\end{equation}
Since $K\geq0$, we have $K^2\leq kK$.

For each $s\in\{1,-1\}$, put
\[
 R=sA-aI,\qquad
 \mathcal M_s=\overline{\Ran(|R|+R)}.
\]
The operators $R$ and $|R|$ commute with $|R|+R$, so $\mathcal M_s$
reduces $A$ and hence also $D=|A|$. The relations
\[
 (R-|R|)(R+|R|)=0,\qquad
 \mathcal M_s^\perp=\ker(R+|R|)
\]
show that $R=|R|$ on $\mathcal M_s$ and $R=-|R|$ on
$\mathcal M_s^\perp$. Thus,
\begin{equation}\label{eq:subspaces}
 sA|_{\mathcal M_s}\geq aI,
 \qquad sA|_{\mathcal M_s^\perp}\leq aI.
\end{equation}
Since $sA|_{\mathcal M_s}$ is positive, uniqueness of the positive
square root gives
\[
 A|_{\mathcal M_s}=sD|_{\mathcal M_s},
 \qquad D|_{\mathcal M_s}\geq aI.
\]

Fix $s$ and write $\mathcal M=\mathcal M_s$. There is nothing to prove
on this subspace if $\mathcal M=\{0\}$. Otherwise, with respect to
$\mathcal H=\mathcal M\oplus\mathcal M^\perp$, let $S_0$ denote the
upper-left block of an operator $S$. In particular,
\[
 A_0=sD_0,\qquad D_0\geq aI,\qquad B_0=B_0^*,\qquad K_0\geq0.
\]
For $x\in\mathcal M$, we have $\|B_0x\|\leq\|Bx\|$, so
$B_0^2\leq(B^2)_0$. Also, $(K^2)_0\leq kK_0$.
Taking the upper-left block in \eqref{eq:quadratic} therefore gives
\begin{equation}\label{eq:compression}
 B_0^2+i[A_0,B_0]+D_0K_0+K_0D_0-kK_0\leq0.
\end{equation}

Now, put
\[
 H=D_0-\frac{k}{2}I,\qquad X=sB_0.
\]
Then, $H\geq(a-k/2)I>0$, and \eqref{eq:compression} becomes
\[
 X^2+i[H,X]+HK_0+K_0H\leq0.
\]
Lemma~\ref{lem:inequality}, applied on $\mathcal M$, gives $K_0=0$.
For $x\in\mathcal M$ it follows that
\[
 \|K^{1/2}x\|^2=\langle Kx,x\rangle
 =\langle K_0x,x\rangle=0.
\]
Thus, $K$ vanishes on both $\mathcal M_1$ and $\mathcal M_{-1}$.

Let
\[
 \mathcal N=(\mathcal M_1+\mathcal M_{-1})^\perp.
\]
This subspace reduces $A$ and $D$. Since $K$ is self-adjoint and
vanishes on $\mathcal M_1+\mathcal M_{-1}$, it is zero on
$\mathcal N^\perp$ and leaves $\mathcal N$ invariant.
On $\mathcal N$, the two inequalities in \eqref{eq:subspaces} give
$-aI\leq A\leq aI$, and hence $0\leq D\leq aI$ on this subspace.
By \eqref{eq:order-bound}, $0\leq K\leq\alpha aI$ on $\mathcal N$.
As $K$ is zero on $\mathcal N^\perp$, the same bound holds on
$\mathcal H$. Consequently,
\[
 k=\|K\|\leq\alpha a<k,
\]
a contradiction. Therefore, we must have $K=0$.

Now \eqref{eq:quadratic} reads $B^2+i[A,B]\leq0$.
With $H=A+(\|A\|+1)I$, we have $H\geq I$ and $[H,B]=[A,B]$.
Lemma~\ref{lem:inequality}, with $X=B$ and $F=0$, gives $B=0$.
\end{proof}

We now deduce our main result from Theorem~\ref{thm:quadratic} by taking
$K=|\re T|-|T|$.

\begin{proof}[Proof of Theorem~\ref{thm:main}]
Write $T=A+iB$ with $A=A^*$ and $B=B^*$, and put $D=|A|$ and
$K=D-|T|$. Then, $0\leq K\leq D$ and $D^2=A^2$. We have
\[
 \begin{aligned}
 A^2+B^2+i[A,B]&=T^*T=|T|^2\\
 &=(D-K)^2=A^2-DK-KD+K^2.
 \end{aligned}
\]
Thus, $B^2+i[A,B]+DK+KD=K^2$.
Theorem~\ref{thm:quadratic}, with $\alpha=1$, gives $B=K=0$,
and therefore $T=T^*$.
\end{proof}

\medskip

\section*{Declarations}

\noindent\textbf{Competing interests.}
The authors declare that they have no competing interests.

\smallskip
\noindent\textbf{Authors Contributions.}
Mohamed Amine Aouichaoui, Fuad Kittaneh, and Yicen Ma contributed equally to this work.

\smallskip
\noindent\textbf{Funding.}
The authors received no funding.

\smallskip
\noindent\textbf{Data availability.}
No datasets were generated during the current study.

\end{document}